\documentclass[11pt]{amsart}

\usepackage{epigamath-voisin}

\usepackage[english]{babel}

\numberwithin{equation}{section}

\usepackage[shortlabels]{enumitem}
\usepackage{amsmath, amsthm, amssymb,amscd}
\usepackage{float}
\usepackage{graphicx}
\usepackage[arrow, matrix, curve]{xy}
\usepackage[all]{xypic}

\usepackage[new]{old-arrows}

\newtheorem{thm}{Theorem}[section]

\newtheorem{cor}[thm]{Corollary}
\newtheorem{prop}[thm]{Proposition}

\theoremstyle{definition}

\theoremstyle{remark}
\newtheorem*{rem}{Remark}

\newcommand{\F}{\mathcal{F}}

\def\PP{{\textbf P}}
\def\OO{\mathcal{O}}

\def\F{\mathcal{F}}

\def\cM{\mathcal{M}}

\def\mm{\overline{\mathcal{M}}}

\DeclareMathOperator{\expdim}{exp.dim}
\DeclareMathOperator{\Pic}{Pic}
\DeclareMathOperator{\Cliff}{Cliff}
\DeclareMathOperator{\Sym}{Sym}
\DeclareMathOperator{\Ker}{Ker}
\DeclareMathOperator{\Coker}{Coker}
\def\irr{\mathrm{irr}}

\newcommand{\supth}[1]{\ensuremath{#1^{\mathrm{th}}}}

\setlist[enumerate,1]{label={\rm(\arabic*)}, ref={\rm\arabic*}}

\YearArticle{2024} \EpigaArticleNr{15} \ReceivedOn{July 27, 2023}
\InFinalFormOn{December 4, 2023}
\AcceptedOn{December 21, 2023}

\title{{Difference varieties and the Green--Lazarsfeld secant conjecture}}
\titlemark{Difference varieties and the Green--Lazarsfeld secant conjecture}

\author{Gavril Farkas}
\address{Humboldt-Universit\"at zu Berlin, Institut f\"ur Mathematik,  Unter den Linden 6, 10099 Berlin, Germany}
\email{farkas@math.hu-berlin.de}

\authormark{G.~Farkas}

\AbstractInEnglish{We establish the Green--Lazarsfeld secant conjecture for general curves of genus $g$ in all the divisorial cases, that is, when the line bundles that fail to be $(p+1)$-very ample form a divisor in the Jacobian of the curve.}


\acknowledgement{The author was supported by the Berlin Mathematics Research Center MATH+ and by the ERC Advanced Grant SYZYGY.
This project has received funding from the European Research Council (ERC) under the European Union Horizon 2020 research and innovation program (grant agreement No. 834172). }

\dedication{To Claire Voisin, with admiration and friendship}

\begin{document}


\maketitle

\begin{prelims}

\DisplayAbstractInEnglish

\bigskip

\DisplayKeyWords

\medskip

\DisplayMSCclass

\end{prelims}


\newpage

\setcounter{tocdepth}{1}

\tableofcontents


\section{Introduction}

For a smooth curve $C$ of genus $g$ and a line bundle $L$ on $C$, following Green, \textit{cf.} \cite{green-koszul}, the Koszul cohomology group $K_{p,q}(C,L)$ of $p$-syzygies of weight $q$ is obtained from the minimal free graded $\Sym\ H^0(C,L)$-resolution of the coordinate ring
$$
\Gamma_C(L):=\bigoplus_{q\in \mathbb Z} H^0\left(C, L^{q}\right).
$$
We write $b_{p,q}(C,L):=\dim\ K_{p,q}(C,L)$ for the number of $p$-syzygies of weight $q$ of the embedded curve $\phi_L\colon C\rightarrow \PP^r$.

Green's conjecture, \textit{cf.} \cite{green-koszul}, characterizing the vanishing of the Betti numbers of a canonical curve $C\subseteq \PP^{g-1}$ in terms of the Clifford index $\Cliff(C)$ of the curve has probably been the most influential statement in the theory of syzygies. Even though Green's conjecture for arbitrary smooth curves remains open, several solutions have been found in the case of generic curves, starting with Voisin's landmark papers \cite{voisin-even, voisin-odd} in the early 2000s, continuing more recently with the approach via Koszul modules in \cite{AFPRW} (which has the benefit of proving the statement in positive characteristic as well, establishing a conjecture of Eisenbud and Schreyer, \textit{cf.} \cite{ES}), and finishing with Kemeny's simpler proof using $K3$ surfaces; \textit{cf.} \cite{K2}. Even more recently, alternative proofs of each of these new approaches to the generic Green's conjecture have been put forward; \textit{cf.} \cite{RS,Ra,Pa}.

The \emph{secant conjecture} proposed by Green and Lazarsfeld, \textit{cf.} \cite{GL1}, is a generalization of Green's conjecture to the case of line bundles of not too low degree. It predicts that if $L\in \Pic^d(C)$ is a line bundle on a curve $C$ of genus $g$ such that
\begin{equation}\label{ineq:gl}
d\geq 2g+p+1-2h^1(C,L)-\Cliff(C),
\end{equation}
then one has $K_{p,2}(C,L)= 0$ if and only if $L$ is $(p+1)$-very ample.\footnote{Recall that a line bundle $L$ on $C$ is $(p+1)$-very ample if every effective divisor of degree $p+2$ on $C$ imposes independent conditions on the linear system $|L|$.} Disposing quickly of the case $H^1(C,L)\neq 0$, which is well known to be equivalent to Green's conjecture for $C$, for non-special line bundles, the secant conjecture can be reformulated as the following equivalence:
\begin{equation}\label{conj:GL}
K_{p,2}(C,L)\neq 0\  \Longleftrightarrow \  L-\omega_C\in C_{p+2}-C_{2g-d+p}.
\end{equation}

Here $C_b-C_a\subseteq \Pic^{b-a}(C)$ denotes the \emph{difference variety} of line bundles $\OO_C(D_b-D_a)$, where $D_b$ and $D_a$ are effective divisors of degrees $b$ and $a$ on $C$. The secant conjecture has been established in all degrees for a general pair $(C,L)$ in
\cite[Theorem 1.3]{FK1}. In the case $\deg(L)=2g+p+1-c$, there are complete results for an \emph{arbitrary} smooth curve $C$ when $c=1$, see \cite[Theorem 2]{GL1}, and when $c=2$, as long as $C$ is not bielliptic; see \cite{Ag}. There are furthermore complete results in the case $d=2g$ and when $g$ is odd and $C$ has maximal gonality; see \cite[Theorem 1.4]{FK1}.

In this paper we establish the secant conjecture for curves of genus $g$ in the divisorial case $d=g+2p+3$, that is, when the difference variety
$C_{p+2}-C_{2g-d+p}$ on the right-hand side of the conjectured equivalence (\ref{conj:GL}) describes a divisor in the corresponding Jacobian variety.

\begin{thm}\label{thm:GL}
Let $C$ be a smooth curve of genus $g$, and fix $\left\lceil \frac{g-3}{2} \right\rceil\leq p \leq g-3$. Assume
$$\dim\ W^1_{p+2}(C)=2p-g+2.$$ One has the following equivalence for a line bundle $L\in \Pic^{g+2p+3}(C)$:
$$K_{p,2}(C,L)=0 \ \Longleftrightarrow \ L-\omega_C\notin C_{p+2}-C_{g-p-3}.$$
\end{thm}

Note that Theorem~\ref{thm:GL}  is of interest only when $p\leq g-3$; else $d\geq 2g+p+1$, in which case the vanishing $K_{p,2}(C,L)=0$ follows automatically  from Green's result \cite[Theorem 4.a.1]{green-koszul}. Theorem~\ref{thm:GL} establishes the  secant conjecture in its strongest form for \emph{general} curves in the divisorial case.

\begin{cor}
  We fix an integer $\left\lceil \frac{g-3}{2} \right\rceil \leq p\leq g-3 $. Then the Green--Lazarsfeld secant conjecture holds for a general curve $C$ and for an arbitrary line bundle $L\in \Pic^{g+2p+3}(C)$.
\end{cor}

An important ingredient in the proof of Theorem~\ref{thm:GL} is provided by the syzygetic interpretation of the divisorial difference varieties given in \cite{FMP}; for every smooth curve $C$ of genus $g$ and $p\leq g-3$, one has the following identification of cycles on $\Pic^{2p+5-g}(C)$:
\begin{equation}\label{eq:diffvar}
C_{p+2}-C_{g-p-3}=\left\{\xi\in \Pic^{2p+5-g}(C): H^0\left(C, \bigwedge^{p+2} M_{\omega_C}\otimes \omega_C\otimes \xi\right)\neq 0\right\},
\end{equation} where $M_{\omega_C}:=\Ker\{H^0(C,\omega_C)\otimes \OO_C \rightarrow \omega_C\}$  is the syzygy (or kernel) bundle on $C$. Note that the slope of the vector bundle $\bigwedge^{p+2} M_{\omega_C}\otimes \omega_C\otimes \xi$ is equal to $g-1$; therefore, the right-hand side of (\ref{eq:diffvar}) is indeed expected to be a divisor on $\Pic^{2p+5-g}(C)$. Using (\ref{eq:diffvar}), coupled with Green's duality theorem \cite[Theorem 2.c.6]{green-koszul}
$$K_{p,2}\left(C,L\right)^{\vee}\cong K_{r(L)-p-1,0}\left(C, \omega_C, L\right),$$
where $r(L)=\deg(L)-g=2p+3$,
Theorem~\ref{thm:GL} can be reformulated as a form of \emph{strange duality} for mixed Koszul cohomology groups. Using the notation of \cite{green-koszul}, we recall that for line bundles $L, L'$ on a curve $C$, the Koszul cohomology group $K_{p,q}(C,L',L)$ is obtained from the minimal resolution of the $\Sym\  H^0(C,L)$-module $\bigoplus_{q\in \mathbb Z} H^0\left(C, L'\otimes L^{q}\right)$.

\begin{cor}\label{cor:strange}
Let $C$ be a smooth curve of genus $g$, and fix $\left \lceil \frac{g-3}{2}\right \rceil \leq p\leq g-3$. Assuming\, $\dim\ W^1_{p+2}(C)=2p-g+2$, the following equivalence holds for a line bundle $L\in \Pic^{g+2p+3}(C)$:
$$K_{p+2,0}\left(C,\omega_C,L\right)=0\  \Longleftrightarrow \ K_{p+2,0}\left(C, L, \omega_C\right )=0.$$
\end{cor}

I do not currently have a direct geometric explanation of the equivalence provided in Corollary~\ref{cor:strange}, that is, one that does not go through the main identification of \cite{FMP}.

The proof of Theorem~\ref{thm:GL} relies in an essential way on the case $d=2g$ of the secant conjecture handled in \cite[Theorem 1.4]{FK1}. When $d=2g$ and $g=2p+3$, via a comparison of two effective divisors of the universal Picard stack $\mathfrak{Pic}_g^{2g}$ parametrizing pairs $[X, L_X]$, where $X$ is a smooth curve of genus $g$ and $L_X$ is a line bundle on $X$ of degree $\deg(L_X)=2g$, for \emph{every} curve $[X]\in \cM_{g}$ having maximal gonality $p+3$, one has the following equivalence:
\begin{equation}\label{eq:equiv}
K_{p,2}(X, L_X)\neq 0 \Longleftrightarrow L_X-\omega_X\in X_{p+2}-X_{p}.
\end{equation}

We explain our strategy for proving Theorem~\ref{thm:GL} in odd genus, the case of even genus being similar. Starting with a  curve $C$ of genus $g=2i+1$, where $i\leq p+1$, and with a line bundle $L\in \Pic^{2i+2p+4}(C)$ such that $K_{p,2}(C,L)\neq 0$, the challenge is to manufacture a $(p+2)$-secant $p$-plane of the embedded curve
$\phi_L\colon C\rightarrow \PP^{2p+2}$. To that end, we set $\ell:=p+1-i\geq 0$, attach to $C$ general $2$-secant rational curves $R_1, \ldots, R_{2 \ell}$ and consider the following semistable curve:
$$X:=C\cup R_1\cup \cdots \cup R_{2 \ell}.$$
Note that $g(X)=2p+3=2i+2\ell+1$. The Brill--Noether assumption on $C$ implies that $X$ has maximal gonality $p+3$.  We consider a line bundle $L_X$ on $X$ whose  restrictions to the components of $X$ are given by
$$L_{X|C}=L \quad\text{and}\quad L_{X|R_j}\cong \OO_{R_j}(1),\quad \text{for}\ j=1, \ldots, 2\ell.$$
In particular, $\deg(L_X)=4p+6=2p_a(X)$, and the assumption $K_{p,2}(C,L)\neq 0$ quickly implies that $K_{p,2}(X,L_X)\neq 0$ as well. Using a suitable extension of (\ref{eq:equiv}) that covers the case of the nodal curve $X$ (and which has appeared in different guises in \cite{Ap} or \cite{K1}), coupled with the main result of \cite{FMP} which enables one to express the difference variety on the left-hand side of (\ref{eq:equiv}) in a way that makes sense on singular curves as well, we obtain that
$$H^0\left(X, \bigwedge^{p} M_{\omega_X}\otimes \omega_X^{2}\otimes L_X^{\vee}\right)=0.$$
Varying the points of intersection of the rational curves $R_j$ with $C$, we obtain via a homological argument that for an integer $0\leq j\leq 2\ell+1$,  the inclusion
\begin{equation}\label{eq:intro_diff1}
L-\omega_C-C_{2(2\ell-j+1)}\subseteq C_{i+j-\ell}-C_{i+\ell-j}
\end{equation}
holds, where the left-hand side is regarded as a divisor in the Jacobian $\Pic^{2j-2\ell}(C)$. An argument using secant varieties on curves then leads us  to conclude that necessarily $L-\omega_C\in C_{i+\ell+1}-C_{i-\ell-1}$, which  corresponds precisely to the inclusion (\ref{eq:intro_diff1}) in the extremal case $j=2\ell+1$. This is equivalent to $\phi_L\colon C\rightarrow \PP^{2p+2}$ having a $(p+2)$-secant $p$-plane.

\subsection*{Acknowledgments}
I profited from discussions with M.~Aprodu, M.~Kemeny and C.~Voisin on this circle of ideas.  I thank the referee for a very careful reading of this paper.

\section{Syzygies of curves}

In this section we discuss a number of basic facts on syzygies of curves that will be used throughout the paper. If $X$ is a projective variety and  $L\in \Pic(X)$ is a globally generated line bundle on $X$, the Koszul cohomology group $K_{p,q}(X, L)$ is by definition the cohomology of the complex
\begin{align*}
\bigwedge^{p+1} H^0(X,L)\otimes H^0\left(X, L^{q-1}\right)\xrightarrow{d_{p+1,q-1}} \bigwedge^p H^0(X,L)\otimes H^0\left(X, L^q\right) \\
\xrightarrow{d_{p,q}} \bigwedge^{p-1} H^0(X,L)\otimes H^0\left(X,  L^{q+1}\right),\end{align*}
where $d_{p,q}$ is the Koszul differential. With the help of the \emph{syzygy bundle}
$$M_L:=\Ker\left\{H^0(X,L)\otimes \mathcal{O}_X\longrightarrow L\right\},$$
one has the following  interpretation of the Koszul cohomology group:
\begin{equation}\label{eq:kernel_bundle}
K_{p,q}(X, L) \cong \Coker \left\{\bigwedge^{p+1} H^0(X,L)\otimes H^0(X, L^{q-1})\longrightarrow H^0\left(X, \bigwedge^p M_L\otimes L^q\right)\right\}.
\end{equation}

We also use the notation $Q_L:=M_L^{\vee}$ for the dual of the syzygy bundle. If $L$ is a non-special line bundle, using (\ref{eq:kernel_bundle}) we have the following equivalence:
$$K_{p,2}(X,L)=0\Longleftrightarrow H^1\left(X, \bigwedge^{p+1} M_L\otimes L\right)=0.$$

We will use the setting when $L$ is a globally generated line bundle on a smooth curve $C$ and $X:=C\cup R_1\cup \cdots \cup R_a$ is a nodal curve obtained from $C$ by attaching $2$-secant mutually disjoint rational curves $R_j$ meeting $C$ at general points $x_j, y_j$. Since the restriction map $\Pic(X)\rightarrow \Pic(C)$ is surjective, we can choose a line bundle $L_X\in \Pic(X)$  such that
$$L_{X| C}\cong L \quad\text{and}\quad   L_{X|R_j}\cong \OO_{R_j}(1), \quad \text{for }\  j=1, \ldots, a.$$

\begin{prop}\label{prop:degen}
One has a natural surjection $K_{p, 2}(X, L_X)\twoheadrightarrow K_{p,2}(C,L)$.
\end{prop}
\begin{proof}  Set $r:=h^0(C,L)-1$. From the Mayer--Vietoris sequence on $X$
  $$
  0\longrightarrow H^0(X,L_X)\longrightarrow  H^0(C,L)\oplus \left(\bigoplus_{j=1}^a H^0\left(R_j, \OO_{R_j}(1)\right)\right)\longrightarrow \bigoplus_{j=1}^a \mathbb C^2_{x_j, y_j},
  $$
we obtain that $H^0(X,L_X)\cong H^0(C,L)$. In the same way, one has natural injections $H^0(C, \omega_C)\hookrightarrow H^0(X, \omega_X)$ and
$H^0\left(C, \omega_C\otimes L\right)\hookrightarrow  H^0\left(X, \omega_X\otimes L_X\right)$. Next we apply the duality theorem
$$K_{p,2}(X, L_X)\cong K_{r-p-1,0}(X, \omega_X, L_X)^{\vee},$$
together with the  commutative diagram
\[
\xymatrixcolsep{5pc}
\xymatrix{
  \bigwedge^{r-p-1} H^0(C, L)\otimes H^0(C, \omega_C)\ar[d]_{j}  \ar[r]^{d_{r-p-1,0}^C} & \bigwedge^{r-p-2} H^0(C,L)\otimes H^0(C,\omega_C\otimes L)\ar[d]_{\bar{j}} &\\
\bigwedge^{r-1-p} H^0(X,L_X)\otimes H^0(X,\omega_X)\ar[r]^{d_{r-1-p,0}^X} &  \bigwedge^{r-p-2} H^0(X,L_X)\otimes H^0(X,\omega_X\otimes L_X)\rlap{,} &
}
\]
where $d_{r-p-1,0}^C$ and $d_{r-p-1,0}^X$ are the Koszul  differentials for $C$ and $X$, respectively.
Observe that both  $j$ and $\bar{j}$ are injective, which implies  that $\ker(d_{r-1-p,0}^C)\hookrightarrow \ker(d_{r-1-p,0}^X)$ is injective, which, after dualizing, finishes the proof.
\end{proof}

\subsection{The divisorial difference varieties in Jacobians}
For a smooth curve $C$ and integers $a, b\geq 0$, we denote by $C_b-C_a\subseteq \Pic^{b-a}(C)$ the difference variety consisting of all line bundles of the form $\OO_C(D_b-D_a)$, where $D_b$ and $D_a$ are effective divisors of degree $b$, respectively $a$, on $C$. In the case when the difference varieties are divisors in the respective Jacobians, that is, when $b=g-a-1$, the main result of \cite{FMP} provides an alternative description of the divisorial difference variety in terms of the (non-abelian) theta divisor of the vector bundle $\bigwedge^a Q_{\omega_C}$, precisely
\begin{equation}\label{eq:fmp}
C_{g-a-1}-C_a=\Theta_{\bigwedge^a Q_{\omega_C}}=\left\{\xi\in \Pic^{g-2a-1}(C): h^0\left(C, \bigwedge^a Q_{\omega_C}\otimes \xi\right)\geq 1\right\}.
\end{equation}

We shall also make use of the varieties of secant divisors for a linear system on a curve. Given a line bundle $L$ on $C$ with $h^0(C, L)=r+1$, for integers $e>0$ and $0\leq f <e$, we introduce the variety of $e$-secant $(e-f-1)$-divisors with respect to the complete linear system $|L|$
\begin{equation}\label{eq:secant_var}
V_e^{e-f}(L):=\left\{D\in C_e: \dim\ |L(-D)|\geq r-e+f\right\}.
\end{equation}
The expected dimension of $V_e^{e-f}(L)$ as a determinantal subvariety of $C_e$ is equal to
$$\expdim\ V_e^{e-f}(L)=e-f(r+1-e+f).$$ For various results on the structure of the secant loci $V_e^{e-f}(L)$, we refer to \cite[Section~8]{ACGH}, \cite{AS} and \cite{Fa}.

The Green--Lazarsfeld secant conjecture (\ref{conj:GL}) offers a characterization of those line bundles $L\in \Pic^d(C)$ satisfying the condition $K_{p,2}(C,L)\neq 0$ in terms of secant varieties of linear systems, predicting the following equivalence:
$$K_{p,2}(C,L)\neq 0 \ \Longleftrightarrow \ V_{p+2}^{p+1}(L)\neq \emptyset,$$
as long as $H^1(C,L)=0$ and $d\geq 2g+p+1-\Cliff(C)$. Note that via a projection argument, it is immediate to see that if $V_{p+2}^{p+1}(L)\neq 0$, then $K_{p,2}(C,L)\neq 0$; therefore, the secant conjecture concerns the reverse implication.
The divisorial case of the secant conjecture treated in this paper  corresponds to the situation
$$\expdim\  V_{p+2}^{p+1}(L)=-1\Longleftrightarrow d=g+2p+3.$$

A solution  to the secant conjecture that holds for \emph{every} curve of maximal gonality is provided in \cite{FK1} in the particular divisorial case $d=2g$ when the genus $g=2p+3$ is odd; see the equivalence (\ref{eq:equiv}). The aim of this paper is to establish the secant conjecture for \emph{all} divisorial cases and for curves of arbitrary gonality, at the price of allowing the base curve $C$ to satisfy certain (mild) generality assumptions of Brill--Noether nature.

We fix an odd genus $g=2p+3$ and recall that $\pi\colon \overline{\mathfrak{Pic}}^{2g}_{g}\rightarrow \mm_{g}$ denotes Caporaso's compactification, \textit{cf.} \cite{Ca}, of the universal Jacobian $\mathfrak{Pic}^{2g}_{g}\rightarrow \cM_{g}$ of degree $2g$ over the moduli space
of curves of genus~$g$. A point of $\overline{\mathfrak{Pic}}^{2g}_{g}$ corresponds to a pair $[X, L_X]$, where $X$ is a quasi-stable curve of genus $g$ and $L_X$ is a \emph{balanced} line bundle of total degree $2g$ on $X$; see \cite[Section~3]{Ca}.

\begin{prop}\label{prop:div}
Let $X$ be a quasi-stable curve of genus $g=2p+3$ having no disconnecting nodes, and let $L_X$ be a globally generated balanced bundle of degree $2g$ on $X$. Assume $X$ has gonality $p+3$. Then the following equivalence holds:
$$K_{p,2}(X,L_X)=0\ \Longleftrightarrow \ H^0\left(X, \bigwedge^p Q_{\omega_{X}}\otimes L_X\otimes \omega_X^{\vee}\right)=0.$$
\end{prop}
\begin{proof}
We denote by $\widetilde{\mathfrak{Syz}}$ the effective divisor on $\overline{\mathfrak{Pic}}^{2g}_{g}$ consisting of pairs $[X, L_X]$ with $K_{p,2}(X, L_X)\neq 0$. The determinantal structure giving rise to $\widetilde{\mathfrak{Syz}}$ can be read off from \cite[Section~6]{FK1}. We write
$\mathfrak{Syz}:=\widetilde{\mathfrak{Syz}}\cap \mathfrak{Pic}^{2g}_{g}$. Similarly, we consider the \emph{secant divisor}
$$\widetilde{\mathfrak{Sec}}:=\left\{[X,L_X]\in \overline{\mathfrak{Pic}}^{2g}_{g}: H^0\left(\bigwedge^p Q_{\omega_X}\otimes L_X\otimes \omega_X^{\vee}\right)\neq 0\right\}$$
and set $\mathfrak{Sec}:=\widetilde{\mathfrak{Sec}}\cap \mathfrak{Pic}^{2g}_{g}$. Using \cite[Theorem 1.4]{FK1}
and the identification (\ref{eq:diffvar}) of divisorial varieties, we conclude that we have the following set-theoretic equality of divisors on $\overline{\mathfrak{Pic}}^{2g}_g$:
$$\overline{\mathfrak{Syz}}=\pi^{-1}\left(\mm_{g,p+2}^{1}\right) \cup \overline{\mathfrak{Sec}},$$
where $\mm_{g,p+2}^1$ is the Hurwitz divisor consisting of stable curves of gonality at most $p+2$. The closure in both sides of this equality is taken inside $\overline{\mathfrak{Pic}}_g^{2g}$. To conclude,  since $\Delta_{\irr}$ is the only boundary divisor of $\mm_g$ possibly containing the point $[X]\in \mm_g$, it suffices to show that $\widetilde{\mathfrak{Syz}}$ does not contain the boundary divisor $\pi^{-1}\left(\Delta_{\irr}\right)$, where $\Delta_{\irr}$ denotes the closure in $\mm_g$ of the locus of irreducible curves.
 Since $\pi^{-1}\left(\Delta_{\irr}\right)$ is irreducible (see \cite[Theorem 3.2]{MV}), it suffices to provide \emph{one example} of a one-nodal curve $Y$ of genus $2p+3$ and a line bundle $L_Y\in \Pic^{4p+6}(Y)$ such that $K_{p,2}(Y,L_Y)=0$. This follows from \cite[Theorem 1.8]{FK1}, where such a vanishing is provided for all curves lying in an ample linear system on a $K3$ surface. Since such a linear system contains one-nodal irreducible curves, the conclusion follows.
\end{proof}

\section{The secant conjecture and divisorial difference varieties}

In this section we prove Theorem~\ref{thm:GL}. We start with a smooth curve $C$ of genus $g$  and a non-special line bundle $L\in \Pic^{g+2p+3}(C)$, where $\left \lceil \frac{g-3}{2}\right \rceil\leq p\leq g-3$. If $c:=\Cliff(C)$, since $W^1_{c+2}(C)+C_{p-c}\subseteq W^1_{p+2}(C)$, the assumptions on $C$ ensure that  inequality (\ref{ineq:gl}) is satisfied.

The proof of Theorem~\ref{thm:GL} will depend on the parity of $g$. We first treat  that the case of odd genus, when we write
\begin{equation}\label{eq:nota}
g=2i+1,\quad\text{with}\ i\leq p+1.
\end{equation}

We set $\ell:=p+1-i\geq 0$. We pick $2\ell$ pairs of general points $(x_j, y_j)$ on $C$, and we introduce the curve
\begin{equation}\label{def:X}
X:=C\cup_{\{x_1, y_1\}} R_1 \cup \cdots \cup_{\{x_{2\ell}, y_{2\ell}\}} R_{2\ell},
\end{equation}
where each $R_j$ is a smooth rational curve meeting $C$ transversally at the points $x_j$ and $y_j$, for $j=1, \ldots, 2\ell$. The curves $R_{j'}$ and $R_j$ are disjoint for $j\neq j'$.  Note that $X$ is a quasi-stable curve of genus $g(X)=2i+2\ell+1=2p+3$. We further introduce a line bundle $L_X\in \Pic^{4p+6}(X)$ such that
$$L_{X|C} \cong L \quad\text{and}\quad L_{|R_j}\cong \OO_{R_j}(1), \quad\text{for }\  j=1, \ldots, 2\ell.$$
Note that indeed $\deg(L_X)=\deg(L)+2\ell=2i+2p+4+2\ell=4p+6$. Observe, furthermore, that $L_X$ is a \emph{balanced} line bundle on $X$ in the sense of Caporaso; \textit{cf.}  \cite{Ca}. In particular, the pair $[X, L_X]$ can be regarded as a point in the compactification $\overline{\mathfrak{Pic}}^{4p+6}_{2p+3} \rightarrow \mm_{2p+3}$ of the universal Jacobian $\mathfrak{Pic}_{2p+3}^{4p+6}$ constructed in \cite{Ca}. From the generality assumptions on $C$ and on the points $(x_j, y_j)$, we obtain that the stabilization of $[X]\in \mm_{2p+3}$ (obtained by contracting the rational curves $R_j$) is a stable curve of maximal gonality; that is, it does not lie in the Hurwitz divisor $\mm_{2p+3, p+2}^1$ of (stable) curves of gonality at most $p+2$ (see also \cite{Ap, FK2} for variations of this argument).

\begin{proof}[Proof of Theorem~\ref{thm:GL}] We keep the same notation as above and assume $K_{p,2}(C,L)\neq 0$, and we aim to show that $L-\omega_C\in C_{i+\ell+1}-C_{i-\ell-1}$. Applying Proposition~\ref{prop:degen}, we obtain that $K_{p,2}(X,L_X)\neq 0$. The assumption $\dim\ W^1_{p+2}(C)=2p-g+2$ implies that when choosing the points $x_j, y_j\in C$ generally,  $[X]\notin \mm_{2p+3,p+2}^1$, that is, $X$ is not a limit of smooth curves of genus $2p+3$ and gonality $p+2$. We are therefore in a position to apply Proposition~\ref{prop:div}, to obtain that
$$H^0\left(\bigwedge^p Q_{\omega_X}\otimes L_X\otimes \omega_X^{\vee}\right)\neq 0.$$
By Riemann--Roch on $X$, this  condition is equivalent to
$$ H^1\left(\bigwedge^p Q_{\omega_X}\otimes L_X\otimes \omega_X^{\vee}\right)\neq 0,$$
and by  Serre duality, this last condition translates into the non-vanishing statement
\begin{equation}\label{eq:non_van2}
H^0\left(X, \bigwedge^p M_{\omega_X}\otimes \omega_X^2\otimes L_X^{\vee}\right)\neq 0.
\end{equation}

We denote by $D_{2\ell}:=x_1+y_1+\cdots+x_{2\ell}+y_{2\ell}\in C_{4\ell}$ the divisor of marked points on $C$. In order to make the condition (\ref{eq:non_van2}) explicit, we note that since $\omega_{X|R_j}\cong \OO_{R_j}$, one has $M_{\omega_X|R_j} \cong \OO_{R_j}^{\oplus (2p+2)}$ and
$\left(\omega_X^2 \otimes L_X^{\vee}\right)_{|R_j} \cong \OO_{R_j}(-1)$. We write down the following Mayer--Vietoris sequence on $X$:
\begin{align*}
0\longrightarrow H^0\left(X, \bigwedge^p M_{\omega_X}\otimes \omega_X^2\otimes L_X^{\vee}\right)\longrightarrow
H^0\left(C, \bigwedge^p M_{\omega_X|C}\otimes \omega_C^2\otimes L^{\vee}(2D_{2\ell})\right)\oplus \\ \oplus  \ \bigoplus_{j=1}^{2\ell} H^0\left(R_j, \left(\bigwedge^p  \OO_{R_j}^{\oplus (2p+2)}\right)(-1)\right)\longrightarrow
H^0\left(\bigwedge^p M_{\omega_X}\otimes \omega_X^2\otimes L_X^{\vee} \otimes \OO_{D_{2\ell}}\right)\longrightarrow \cdots.
\end{align*}
Since $H^0\left(R_j, \bigwedge^p M_{\omega_X|R_j}\otimes \omega_{X|R_j}^2\otimes L_{R_j}^{\vee}\right)=0$ for each $j=1, \ldots, 2\ell$, we obtain an inclusion
$$
H^0\left(X, \bigwedge^p M_{\omega_X}\otimes \omega_X^2\otimes L_X^{\vee}\right)\longhookrightarrow H^0\left(C, \bigwedge^p M_{\omega_X|C}\otimes \omega_C^2\otimes L^{\vee}(D_{2\ell})\right);
$$
therefore, our hypothesis (\ref{eq:non_van2}) implies that the following equivalent statements also hold:
\begin{equation}\label{eq:non-van3}
H^0\left(C, \bigwedge^p M_{\omega_X|C}\otimes \omega_C^2\otimes L^{\vee}(D_{2\ell})\right)\neq 0\Longleftrightarrow H^1\left(C, \bigwedge^p Q_{\omega_X|C}\otimes L\otimes \omega_C^{\vee}(-D_{2\ell})\right)\neq 0,
\end{equation}
where the equivalence in (\ref{eq:non-van3}) is a consequence of Serre duality.

In order to describe  $M_{\omega_X|C}$, we write down the morphism of short exact sequences
\[
\xymatrixcolsep{4pc}
\xymatrix{
0 \ar[r] & M_{\omega_C} \ar[d]_{} \ar[r] & H^0(C, \omega_C)\otimes \OO_C \ar[d] \ar[r] & \omega_{C} \ar[d]_{} \ar[r] & 0 \\
0 \ar[r]  & M_{\omega_X|C} \ar[r] & H^0(X, \omega_X)\otimes \OO_C \ar[r] & \omega_C(D_{2\ell}) \ar[r] & 0
}
\]
to obtain the following exact sequence of vector bundles on $C$:
$$0\longrightarrow M_{\omega_C}\longrightarrow M_{\omega_X|C}\longrightarrow \bigoplus_{j=1}^{2\ell} \OO_C\left(-x_j-y_j\right) \longrightarrow 0,$$
which, after dualizing, can be rewritten as the exact sequence
\begin{equation}\label{eq:exseq3}
0\longrightarrow \bigoplus_{j=1}^{2\ell} \OO_C\left(x_j+y_j\right)\longrightarrow Q_{\omega_X|C}\longrightarrow Q_{\omega_C}\longrightarrow 0.
\end{equation}

We compute the $\supth{p}$ exterior power of the sequence (\ref{eq:exseq3}). The assumption $p\leq g-3$ translates into $\ell\leq i-1$. There exists a filtration
$$\bigwedge^p Q_{\omega_X|C}=\F^0\supset \F^1\supset \cdots \supset \F^{2\ell}\supset \F^{2\ell+1}=0$$
of vector bundles on $C$ such that the successive quotients are given by
$$\F^{k-1}/\F^{k}\cong \bigwedge^{k-1} \left(\bigoplus_{j=1}^{2\ell} \OO_C(x_j+y_j)\right) \otimes \bigwedge^{p+1-k} Q_{\omega_C}.$$
For each $k=1, \ldots, 2\ell$, we tensor the exact sequence
$$0\longrightarrow \F^k \longrightarrow \F^{k-1} \longrightarrow \bigwedge^{p+1-k} Q_{\omega_C} \otimes \bigwedge^{k-1} \left(\bigoplus_{j=1}^{2\ell} \OO_C\left(x_j+y_j\right)\right)\longrightarrow 0$$
with the line bundle $L\otimes \omega_C^{\vee}(D_{2\ell})$.
Taking cohomology, we obtain that  condition (\ref{eq:non-van3}) implies that there exists an integer $j\leq 2\ell+1$ such that
\begin{equation}\label{eq:non-van4}
H^1\left(C, \bigwedge^{i+\ell-j} Q_{\omega_C}\otimes L\otimes \omega_C^{\vee}(-D_{2\ell-j+1})\right)\neq 0,
\end{equation}
where $D_{2\ell-j+1}$ is an effective divisor of degree $2(2\ell-j+1)$ on $C$ supported on some of the points $x_1, y_1, \ldots, x_{2\ell}, y_{2\ell}$. Since the marked points on $C$ are chosen generally, the divisor $D_{2\ell-j+1}\in C_{2(2\ell-j+1)}$ is general as well; therefore, we obtain that
(\ref{eq:non-van3}) implies the  inclusion
\begin{equation}\label{eq:incl}
L- \omega_C-C_{2(2\ell-j+1)}\subseteq C_{i+j-\ell}-C_{i+\ell-j}
\end{equation}
for some $0\leq j\leq 2\ell+1$. Note that the desired conclusion, that is, $L-\omega_C\in C_{i+\ell+1}-C_{i-\ell-1}$, corresponds precisely to statement  (\ref{eq:incl}) in the extremal case $j=2\ell+1$. We now show that the inclusion (\ref{eq:incl}) for some $j\leq 2\ell+1$ implies statement (\ref{eq:incl}) for $j=2\ell+1$.

We set $\eta:=L\otimes \omega_C^{\vee}\in \Pic^{2\ell+2}(C)$. Condition (\ref{eq:incl}) can be translated into stating that the variety of secant divisors
$$V_{i+\ell-j}^{i+\ell-j-1}\left(\omega_C\otimes \eta^{\vee}\right):=\left\{B\in C_{i+\ell-j}: h^0(C, \eta(B))\geq 1\right\}$$
has dimension at least $2(2\ell-j+1)$, which exceeds by $1$  the expected dimension of the degeneracy locus $V_{i+\ell-j}^{i+\ell-j-1}(\omega_C\otimes \eta)$. We may assume $h^0(C,\eta)=0$; else, since
$$C_{2\ell+2}\subseteq C_{i+\ell+1}-C_{i-\ell-1},$$
the conclusion $\eta\in C_{i+\ell+1}-C_{i-\ell-1}$ is immediate. Therefore, $h^0(C,\eta)=0$, or equivalently, $\omega_C\otimes \eta^{\vee}$ is a non-special line bundle and $h^0(C, \omega_C\otimes \eta^{\vee})=2i-2\ell-2$.

Recalling the definition (\ref{eq:secant_var}) of the varieties of secant divisors, we  apply \cite[Section~VIII.4, p.~356]{ACGH}
to the non-special linear system $|\omega_C\otimes \eta^{\vee}|$ to conclude that $V_{i+\ell-a}^{i+\ell-a-1}(\omega_C\otimes \eta^{\vee})\neq \emptyset$, as long as the following inequality is satisfied:
$$\expdim\ V_{i+\ell-a}^{i+\ell-a-1}\left(\omega_C\otimes \eta^{\vee}\right)=4\ell-2a+1\geq  0\Longleftrightarrow a\leq 2\ell.$$

In particular, setting $a=2\ell$, we find $V_{i-\ell}^{i-\ell-1}\left(\omega_C\otimes \eta^{\vee}\right)\neq \emptyset$, and applying the result of Aprodu and Sernesi \cite[Theorem 4.1]{AS} on secant varieties of excessive dimension, we obtain that (\ref{eq:incl}) implies the estimate
$$\dim\ V_{i-\ell}^{i-\ell-1}\left(\omega_C\otimes \eta^{\vee}\right)\geq \dim\ V_{i+\ell-j}^{i+\ell-j-1}\left(\omega_C\otimes \eta^{\vee}\right)-2(2\ell-j)=2.$$

Thus $V_{i-\ell}^{i-\ell-1}(\omega_C\otimes \eta^{\vee})=\{D\in C_{i-\ell}: H^0(C, \eta(D))\geq 1\}$ is at least $2$-dimensional. Having fixed a base point $p_0\in C$, this implies that the locus
$$Y:=\left\{E\in C_{i-\ell-1}: H^0\left(C, \eta(E+ p_0)\right)\neq 0\right\}$$
has dimension at least $1$. We consider the Abel--Jacobi map $\theta\colon Y\rightarrow \Pic^{i+\ell+2}(C)$  given by $\theta(E):=\eta(E+p_0)$.
If $\theta$ is generically finite, it follows that $\dim\ \theta(Y)\geq1$. Since the locus $p_0+C_{i+\ell+1}$ is an \emph{ample} divisor in $C_{i+\ell+2}$, it follows that $\theta(Y)$ intersects $p_0+C_{i+\ell+1}$; therefore, there exist $E\in C_{i-\ell-1}$ and $E'\in C_{i+\ell+1}$ such that $\eta(E+p_0)=\OO_C(p_0+E')$, amounting to
\begin{equation}\label{eq:conclusion}
L-\omega_C\in C_{i+\ell+1}-C_{i-\ell-1},
\end{equation}
as desired. If on the other hand, $\theta$ is not generically finite, then $h^0\left(C, \eta(E+p_0)\right)\geq 2$, for every $E\in Y$, in which case, once again, one finds an effective divisor $E'\in C_{i+\ell+1}$ such that $\eta(E+p_0)=\OO_C(p_0+E')$, which once more implies the conclusion (\ref{eq:conclusion}). This finishes the proof of Theorem~\ref{thm:GL} in the odd genus case.
\end{proof}

\begin{proof}[Proof of Theorem~\ref{thm:GL} in the case of even genus] The proof follows with minor modification along the lines of the odd genus case. We fix a smooth curve $C$ of genus $g=2i$ satisfying the condition
$$\dim\ W^1_{p+2}(C)=g-2p+2$$ and a line bundle $L\in \Pic^{2i+2p+3}(C)$ such that $K_{p,2}(C,L)\neq 0$. Setting $\ell:=p+1-i\geq 0$ once more, the aim is to show that
$$L-\omega_C\in C_{i+\ell+1}-C_{i-\ell-2}.$$ Note that in this case the assumption $d\leq 2g+p$ translates into $\ell\leq i-2$.

This time we attach $2\ell+1$ rational curves $R_1, \ldots, R_{2\ell+1}$ to $C$ such that  each $R_j$ meets $C$ at two general points $x_j, y_j\in C$. The resulting curve $X$ is a quasi-stable curve of genus $g(X)=2p+3$ and gonality $p+3$. We choose a line bundle $L_X\in \Pic^{4p+6}(X)$ whose restriction to $C$ is the line bundle $L$, whereas $L_{X|R_j}\cong \OO_{R_j}(1)$, for $j=1, \ldots, 2\ell+1$. Following \emph{mot \`a mot} the reasoning in the odd genus case, denoting by
$$D_{2\ell+1}:=(x_1+y_1)+\cdots+ (x_{2\ell+1}+y_{2\ell+1})\in C_{4\ell+2}$$
the divisor of marked points, knowing that
$K_{p,2}(X, L_X)\neq 0$, we obtain that
$$H^0\left(X, \bigwedge^p M_{\omega_X|C}\otimes \omega_C^2\otimes L^{\vee}(D_{2\ell+1})\right)\neq 0.$$
Setting $\eta:=L\otimes \omega_C^{\vee}\in \Pic^{2\ell+3}(C)$,  the same filtration argument as in the odd genus case implies there exists a $j\leq 2\ell+2$ such that
$$\dim\ V_{i+\ell-j}^{i+\ell-j-1}\left(\omega_C\otimes \eta^{\vee}\right)\geq 4\ell-2j+4 \ \left(=\expdim\ V_{i+\ell-j}^{i+\ell-j-1}\left(\omega_C\otimes \eta^{\vee}\right)+1\right).$$
The conclusion $\eta\in C_{i+\ell+1}-C_{i-\ell-2}$ follows by applying once more \cite{AS}.
\end{proof}

\subsection{Brill--Noether theory and the secant conjecture}
Using standard results in Brill--Noether theory, for large values of $p$, the condition
\begin{equation}\label{eq:BNdim}
\dim\ W^1_{p+2}(C)=2p-g+2
\end{equation}
appearing in the statement of Theorem~\ref{thm:GL} is always satisfied, and our results are complete in these cases. In the extremal case $p=g-3$, from Martens' theorem \cite[Theorem~IV.5.1]{ACGH}, we obtain that $\dim\ W^1_{g-1}(C)=g-4$ unless $C$ is hyperelliptic. In this case, we recover \cite[Theorem 3.3]{GL2}. In the next case $p=g-4$, applying Mumford's theorem \cite[Theorem~IV.5.2]{ACGH}, we obtain that if $C$ is a non-bielliptic curve with $\Cliff(C)\geq 2$, then the estimate (\ref{eq:BNdim}) is satisfied and Theorem~\ref{thm:GL} reduces to Agostini's result, \textit{cf.} \cite{Ag}, in the case of line bundles of degree $3g-5$. The next cases are already new results, and we record them.

\begin{thm}\label{thm:keem}\leavevmode
\begin{enumerate}
\item Let $C$ be a smooth curve of genus $g$ with $\Cliff(C)\geq 3$. Then one has the following equivalence for a line bundle $L\in \Pic^{3g-7}(C)$:
$$K_{g-5,2}(C,L)=0  \ \Longleftrightarrow \  L-\omega_C \notin C_{g-3}-C_2.$$

\item Let $C$ be a smooth curve of genus $g\geq 12$ such that $\Cliff(C)\geq 4$. Assume that $C$ is not a triple cover of an elliptic curve.
Then one has the following equivalence for $L\in \Pic^{3g-9}(C)$:
$$K_{g-6,2}(C,L)=0 \ \Longleftrightarrow \  L-\omega_C\notin C_{g-4}-C_3.$$
\end{enumerate}
  \end{thm}

\begin{proof} We apply Theorem~\ref{thm:GL}, observing that  assumption (\ref{eq:BNdim}) is satisfied by applying Keem's result \cite[Theorem 2.1]{Ke} when $p=g-5$ and \cite[Corollary 3.3]{Ke} when $p=g-6$.
\end{proof}

For smaller values of $p$, our understanding of which curves satisfy (\ref{eq:BNdim}) is less complete. In the extremal case $p=\frac{g-3}{2}$,  condition (\ref{eq:BNdim}) reduces to saying that $C$ has maximal gonality. In general, it turns out that the sufficient condition from Theorem~\ref{thm:GL} is related to the study of the curves having infinitely many pencils of minimal degree. The following result makes this connection precise.

\begin{prop}\label{prop:bn}
Let $a\geq 0$ and  $C$ be a smooth curve of genus $g\geq (a+1)(2a+1)$ such that $\Cliff(C)\geq a+1$ and $\dim\  W^1_{a+3}(C)<1$. Then the following equivalence holds for a line bundle $L\in \Pic^{3g-2a-3}(C)$:
$$K_{g-3-a,2}(C,L)=0\ \Longleftrightarrow \  L-\omega_C\notin C_{g-1-a}-C_a.$$
\end{prop}

\begin{proof} It suffices to observe that under these hypotheses, applying \cite[Theorem 15]{Co} if $W^1_{a+3}(C)$ is at most zero-dimensional, we necessarily have $\dim\ W^1_{g-1-a}(C)=g-2a-4$, and then  hypothesis (\ref{eq:BNdim}) is satisfied.
\end{proof}

\begin{rem} The first case where Proposition~\ref{prop:bn} may not always apply is for $p=g-7$, where we are not aware of a classification of the smooth $7$-gonal curves $C$ with $\Cliff(C)=5$ and such that $\dim\ W^1_7(C)=1$. One class of curves where condition (\ref{eq:BNdim}) generally fails is that of covers of elliptic curves.  For such curves, however, Kemeny, \textit{cf.} \cite{K3}, recently introduced  novel techniques to understand their syzygies.
\end{rem}


\newcommand{\etalchar}[1]{$^{#1}$}

\end{document}